\documentclass[reqno,10pt]{amsart}
\usepackage[usenames]{color}
\usepackage{amssymb,amsmath,amsthm,amstext,amsfonts, fancybox ,amscd,amsbsy}
\usepackage[dvips]{graphicx}
\usepackage{psfrag}
\usepackage{color}
\usepackage{pst-all}

\makeatletter \@addtoreset{equation}{section} \makeatother

\renewcommand\thetable{\thesection.\@arabic\c@table}

\theoremstyle{plain}
\newtheorem{maintheorem}{Theorem}
\newtheorem{theorem}{Theorem }[section]
\newtheorem{mainproposition}{Proposition}
\newtheorem{lemma}{Lemma}[section]
\newtheorem{corollary}{Corollary}[section]

\newtheorem{remark}[theorem]{Remark}

\newtheorem{defi}{Definition}[section]

\newcommand{\al} {\alpha}

\newcommand{\la} {\lambda}

\def \co {\mathcal{O}}
\def \NN {{\mathbb N}}
\def \tz {{\mathrm{Z}}}

\newcommand{\supp}{\operatorname{supp}}

\newcommand{\dist}{\operatorname{dist}}

\newcommand{\cP}{\mathcal{P}}

\newcommand{\cO}{\mathcal{O}}

\begin{document}

\title{Contribution to the ergodic theory of robustly transitive maps}

\author{C. Lizana}
\address{Departamento de Matem\'atica\\
Facultad de Ciencias\\
Universidad de Los Andes\\
La Hechicera-M\'{e}rida, 5101\\
Venezuela}
 \email{clizana@ula.ve}

\author{V. Pinheiro}
\address{Departamento de Matem\'atica, Universidade Federal da Bahia\\
  Av. Ademar de Barros s/n, 40170-110 Salvador, Brazil.}
\email{vilton@ufba.br}

\author{P. Varandas}
\address{Departamento de Matem\'atica, Universidade Federal da Bahia\\
  Av. Ademar de Barros s/n, 40170-110 Salvador, Brazil.}
\email{paulo.varandas@ufba.br}

\date{\today}

\maketitle

\begin{abstract}
In this article we intend to contribute in the understanding of
the ergodic properties of the set $\mathcal{RT}$ of robustly
transitive local diffeomorphisms on a compact manifold $M$ without
boundary. We prove that there exists a $C^1$
residual subset $\mathcal R_0\subset \mathcal{RT}$ such that any
$f\in \mathcal R_0$ has a residual subset of $M$ with dense
pre-orbits. Moreover, $C^1$ generically in the space of 
local diffeomorphisms with no splitting and all points with dense pre-orbits, there are
uncountably many ergodic expanding invariant measures with full support and
exhibiting exponential decay of correlations.
In particular, these results hold for an important class of robustly
transitive maps considered in~\cite{LP}.
\end{abstract}

\section{Introduction and Statement of the Main Results} \label{sec2}

In the last two decades many advances have been made
to the study of robustly transitive diffeomorphisms, whose geometric
properties are by now very well understood.  In fact,
it follows from Bonatti, D\'iaz, Pujals~\cite{BDP} that robustly
transitive diffeomorphisms exhibit a weak form of hyperbolicity, namely,
 dominated splitting. The ergodic aspects of robustly transitive diffeomorphisms
 have called the attention of many authors recently. For instance,  let us refer to the
 construction of SRB measures (see~\cite{Ca}) and maximal entropy measures
 (see~\cite{BFSV}) for the class of DA-maps introduced by Ma\~n\'e,
 and more recently, \cite{Ures} proved intrinsic
 ergodicity(unique entropy maximizing measure) for partially hyperbolic diffeomorphisms
 homotopic to a hyperbolic one on the 3-torus.

 The theory is much more incomplete in the
 non-invertible setting, in which case the study of robust transitivity has
 received far less attention. Since the negative iterates of an endomorphism
 are not easy to describe, the dynamics can be very hard to explain.
Nevertheless, the first important contributions in this respect were given recently
in~\cite{LC,LP}, where it is shown that there are robustly transitive local diffeomorphisms
without any dominated splitting, and some necessary and sufficient conditions
for robust transitivity of local diffeomorphisms are given.
Our purpose here is to give a contribution to the better understanding of robustly
transitive local diffeomorphisms and to give first results on
their ergodic theory.


Let $M$ be a compact Riemannian manifold. We say that an
endomorphism $f: M\rightarrow M$ is \emph{transitive} if there
exists $x\in M$ such that its forward orbit by $f$,
$\mathcal{O}_f^+(x)=\{f^n(x)\}_{n\geq 0},$ is dense in $M.$
In the theory of differentiable dynamical systems, it is an
important issue to know when a special feature is exhibited in all
nearby systems (with respect to some topology), that is, a
dynamical property is \emph{robust} under perturbation. In
particular, a map $f$ is $C^r$ \emph{robustly transitive} ($r\ge
1$) if there exists  a $C^r$ open neighborhood $\mathcal{U}(f)$ of
$f$ such that every $g\in \mathcal{U}(f)$ is transitive.

We focus our attention to local diffeomorphisms, that is
endomorphisms without critical points. Let us denote by $\mathcal{RT}$ the set of $C^1$
local diffeomorphisms on $M$ that are robustly transitive.
It is not hard to check that any endomorphism that admits a dense subset
of points with dense pre-orbits is transitive. In our first result we address
a sort of converse to the previous assertion for generic maps. More precisely,

\begin{maintheorem}\label{thm:preorbits}
There exists a $C^1$ residual subset $\mathcal R_0\subset
\mathcal{RT}$ such that for any $f\in \mathcal R_0$ the following
conditions hold:
\begin{enumerate}
\item Periodic points are dense in $M$.
\item All periodic orbits are hyperbolic.
\item There exists a residual subset $\mathcal{D} \subset M$ of points such that for any $x \in \mathcal{D}$
the pre-orbit $\co^{-}_{f}(x)=\{w\in f^{-n}(x): n\in \mathbb{N}\}$ is dense in $M$.
\end{enumerate}
\end{maintheorem}

In the remaining,  our goal is to show that robustly transitive local diffeomorphisms are interesting from the
ergodic theory point of view. For that discussion let us recall some necessary definitions.
Given a compact forward invariant set $\Lambda$,
we say that $f\mid_\Lambda$ has \emph{no splitting in a $C^1$ robust way}
if there exists a  $C^1$ open neighborhood $\mathcal{U}(f)$ of $f$ so that for all
$g\in\mathcal{U}(f)$ the tangent space $T_\Lambda M$ does not admit non-trivial
invariant subbundles.
We denote by $\mathcal{RT}^*\subset \mathcal{RT}$ the open subset of $C^1$ robustly transitive
local diffeomorphisms that have no splitting in a $C^1$ robust way.
Let $\mathcal{SRT}^*\subset \mathcal{RT}^*$ denote the set of local diffeomorphisms
so that all points have dense pre-orbits. It is easy to see that $\mathcal{SRT}^*$ has non-empty interior
and we consider in $\mathcal{SRT}^*$ the induced topology from $\mathcal{RT^*}$.
Moreover, an ergodic $f$-invariant probability measure $\mu$ is \emph{expanding}
if all the Lyapunov exponents are positive.
Finally, we say that $(f,\mu)$ has \emph{exponential decay of correlations} if there are constants
$K,\al>0$ and $\lambda \in (0,1)$ such that for all $\psi\in C^\alpha(M,\mathbb R)$, $\varphi\in L^1(\mu)$
and $n\in \mathbb N$:
$$
\Big| \int \psi \, (\varphi\circ f^n) \, d\mu - \int \psi \, d\mu  \; \int \varphi \, d\mu \Big|
    \leq K \lambda^n \|\psi\|_{\alpha} \, \|\varphi\|_{L^1(\mu)}.
$$
Our next result illustrates that generically robustly transitive maps
exhibit many ergodic measures with interesting dynamical meaning.

\begin{maintheorem}\label{thm:measures}
There exists a $C^1$ residual subset $\mathcal R_1\subset
\mathcal{SRT}^*$ such that for any $f\in \mathcal R_1$
there are uncountable many $f$-invariant, ergodic and expanding measures
with full support
and exponential decay of correlations.
\end{maintheorem}

Some comments are in order. We note that, in opposition to the case of diffeomorphisms
discussed in \cite{BDP}, there are open sets of local diffeomorphisms with robust non-existence
of splitting (see Section~\ref{sec5} below). Moreover, there are open subsets of robustly transitive
local diffemorphisms that do not admit invariant expanding measures, e.g. hyperbolic endomorphisms
on $\mathbb T^n$. Clearly, these examples admit some non-trivial invariant subbundles robustly.
The  following proposition, which is interesting by itself, plays a key role  for the proof of Theorem~\ref{thm:measures}.

\begin{mainproposition}\label{prop2.1}
There exists a $C^1$ residual subset $\mathcal R_1\subset
\mathcal{SRT}^*$ such that for every $f\in \mathcal R_1$ the set of hyperbolic periodic points is dense
and it admits a periodic source with dense pre-orbit.
\end{mainproposition}

\begin{remark}
The statements of Theorem~\ref{thm:measures} is a consequence of Theorem~\ref{Theorem8769} together
with Proposition~\ref{prop2.1} above. Moreover, one expects that this proposition holds $C^1$-generically in
$\mathcal{RT}^*$  provided a counterpart of the connecting lemma
(see  \cite{Hay, Wen2}) is given for local diffeomorphisms.
\end{remark}

Concerning our results it is an interesting question to understand if
there are robustly transitive local diffeomorphisms such that the set of periodic
saddle points is dense while the set of periodic sources is non-empty.

The paper is organized as follows. In Section~\ref{sec3} we present
some definitions and prove auxiliary lemmas. In
Section~\ref{sec4} we prove the main results stated in Section~\ref{sec2}.
In Section~\ref{sec5} we present a large class of robustly transitive
local diffeomorphisms which are not uniformly expanding and
exhibit good ergodic properties. Finally, in Section~\ref{sec6} we do
some further comments concerning the  existence of re\-le\-vant expanding measures
assuming the presence of some
kind of dominated splitting.

\section{Robust transitivity and limit sets} \label{sec3}

In this section we prove some preliminary results relating robust
transitivity and existence of dense pre-orbits that play a key role in the
proof of the main results. For that purpose we shall introduce first some definitions.
%
%
Given $\delta>0$, we say that $U\subset M$ is \emph{$\delta$-dense} if $ M \subset \bigcup_{x\in U}B_{\delta}(x)$,
where $B_\delta(x)$ stands for the ball of radius $\delta$ around $x$.
For any endomorphism $f: M \to M$ and $x\in M$,  the \emph{$\omega$-limit set} of a point $x$,
denoted by $\omega_{f}(x),$ is the set of points $y\in M$ such that there exists a
sequence $(n_k)_{k\in \mathbb{N}}$ of positive integers such that $f^{n_k}(x)\rightarrow y$
when $k$ goes to infinity.
Analogously, 
the
\emph{$\alpha$-limit set} of $x$, denoted by $\alpha_{f}(x)$, is the set of accumulation points  $y\in M$
by the pre-orbit of $x$, that is, there exists a sequence $(x_{n_k})_{k\in \mathbb{N}}$ in $\co^{-}_{f}(x)$
satisfying $f^{n_k}(x_{n_k})=x$ and such that $x_{n_k}\rightarrow y$ when $k$ goes to infinity.
Clearly, $\omega_{f}(x)=M$ if and only if the forward orbit of $x$ is dense, and analogous statement also
holds for pre-orbits.
The following lemmas provide a dichotomy of the limit sets for continuous endomorphisms.

\begin{lemma}[Dichotomy of Transitivity]
If $f\in C^{0}( M, M)$ then only one of the following properties hold: either $\omega_{f}(x)\ne M$ for all $x\in M$,
or  the set $\{x\in M \colon \omega_{f}(x)= M\}$ is a residual subset of $M$. 
\end{lemma}

\begin{proof}
Let us suppose that there exists $p\in M$ such that $\co^{+}_{f}(p)$
 is dense in $ M,$ otherwise we are done. Write $p_{\ell}=f^{\ell}(p)$.
Given $n\ge 1$ consider the set $ M_{n}=\{x\in M : \co^{+}_{f}(x)$
is $1/n$-dense$\}$.
By assumption, for each $\ell\in\NN$ and $n\in\NN$, there is some $k_{n,\ell}$ such that
$\{p_{\ell},\cdots,f^{k_{n,\ell}}(p_{\ell})\}$ is $1/2n$-dense.
Moreover, by continuity of $f$ there exists $r_{n,\ell}>0$ such
that $f^{j}(B_{r_{n,\ell}}(p_{\ell}))\subset
B_{1/2n}(f^{j}(p_{\ell}))$ for all $0\le j\le k_{n,\ell}$ and, consequently,
for any $y\in B_{r_{n,\ell}}(p_{\ell})$ it follows that the finite piece of orbit
$\{y,\cdots,f^{k_{n,\ell}}(y)\}$ is $1/n$-dense.
Therefore $\bigcup_{\ell\in\NN}B_{r_{n,\ell}}(p_{\ell})\subset
M_{n}$ is a open and dense set. In particular this proves that
$\bigcap_{n\in\NN}\bigcup_{\ell\in\NN}B_{r_{n,\ell}}(p_{\ell})$ is
a residual subset contained in $\bigcap_{n\in\NN} M_{n}=\{x\in M
:\omega_{f}(x)= M\}$, proving the lemma.
\end{proof}

Given a continuous endomorphism
$f\in C^{0}( M, M)$ we denote by $\mathcal{C}_f$ the set of critical
points of $f$, that is, $x\in \mathcal{C}_f$ if for all $r\!>\!0$ the restriction
$f|_{B_{r}(x)}$ is not a homeomorphism.
The next result relates forward and backward limit sets.

\begin{lemma}\label{le:alpha}
Let $f\in C^{0}( M, M)$ be such that the critical
region $ \mathcal{C}_f$ has empty interior.
If $f$ is transitive then $\{x\in M:
\alpha_{f}(x)=\omega_{f}(x)= M\}$ is a residual subset of $M$.
\end{lemma}

\begin{proof}
The proof mimics the previous lemma with some care with the critical set $\mathcal{C}_f$.
Since $ M\setminus \mathcal{C}_f$ is open and dense, 
then $\bigcap_{j\ge0}f^{-j}( M\setminus \mathcal{C}_f)$ is
residual. It follows from the previous lemma that the intersection
$$
\big\{x\in M\, : \,\omega_{f}(x)= M\big\}
	\,\cap\,\bigg(\bigcap_{j\ge0}f^{-j}( M\setminus \mathcal{C}_f)\bigg)\ne\emptyset
$$
is also a residual subset of $M$. Pick $p\in M$ such that
$\omega_{f}(p)= M$ and $\co^{+}_{f}(p)\cap
\mathcal{C}_f=\emptyset$, and write $p_{j}=f^{j}(p)$. 
%
Given $n\ge 1$ 
consider $A_{n}=\{x\in M : \co^{-}_{f}(x)$ is $1/n$-dense$\}$.
Since $\co^{+}_{f}(p)$ is dense, there exists $k_{n}\in\NN$ such
that $\{p_{0},\cdots,p_{k_{n}}\}$ is $1/2n$-dense. As
$\{p_{0},\cdots,p_{k_{n}}\}\subset\co^{-}_{f}(p_{j})$ for all
$j\ge k_{n},$ we get that $p_{j}\in A_{n}$ for every $j\ge k_{n}$.
Now, using  that $\co_{f}^{+}(p)\cap \mathcal{C}_f=\emptyset,$ 
for each $j\ge k_{n}$ one can find
$r_{j}>0$ such that $f^{j}|_{B_{r_{j}}(p)}$ is a homeomorphism and
$f^{\ell}(B_{r_{j}}(p))\subset B_{\frac{1}{2n}}(f^{\ell}(p))$ for
all $0\le\ell\le j$.
This proves that $\co_{f}^{-}(y)$ is $1/n$-dense for any $y\in
f^{j}(B_{r_{j}}(p))$ and $j\ge k_{n}$. As $f^{j}(B_{r_{j}}(p))$ is
an open neighborhood of $p_{j}$ and $\{p_{j};j\ge k_{n}\}$ is
dense, then $\bigcup_{j\ge k_{n}}f^{j}(B_{r_{j}}(p))\subset A_{n}$
is an open and dense subset of $ M$. Therefore,
$\bigcap_{n\in\NN}\bigcup_{j\ge k_{n}}f^{j}(B_{r_{j}}(p))\subset
\bigcap_{n\in\NN}A_{n}$ is a residual subset.
Since $\bigcap_{n\in\NN}A_{n}=\{x\in M : \alpha_{f}(x)= M\},$ we
notice that $\{x\in M : \alpha_{f}(x)= M\} \cap \{x\in M :
\omega_{f}(x)= M\}$ is a residual subset in $M$. This finishes the
proof of the lemma.
\end{proof}

In particular we obtain the following immediate consequence:

\begin{corollary}\label{cor:dense}
If $f\in\mathcal{RT}$ then there exists a residual subset of
points in $M$ with dense orbit and pre-orbit.
\end{corollary}

In fact, a converse result also holds obtaining that robust density of points with
dense pre-orbit is equivalent to robust transitivity for local diffeomorphisms.

\begin{lemma}
Let $\mathcal{U}$ be an open subset of the space of  $C^1$ local
diffeomorphisms and assume that every $f\in\mathcal U$ admits a
dense set of points with dense pre-orbit. Then every $f\in
\mathcal U$ is robustly transitive, that is, $\mathcal U \subset
\mathcal{RT}$.
\end{lemma}

\begin{proof}
Since the proof is simple we leave it as an easy exercise for the reader.
\end{proof}

Let us mention that expanding endomorphisms are not the only class of maps
satisfying the assumptions of the previous lemma.
In Section~\ref{sec5} we present a class of robustly transitive local diffeomorphisms that are not uniformly
expanding but for which 
there exists a generic subset of points with dense pre-orbit.

\section{Proof of the main results}\label{sec4}

This section is devoted to the proof of our main results.

\subsection{Proof of Theorem~\ref{thm:preorbits}}

Items (1) and (2) are a consequence of the $C^1$ closing lemma
for local diffeomorphisms (see e.g.~\cite{Castro, Mor, Wen}) and Kupka-Smale theorem for
 local diffeomorphisms. Hence, there exists a residual subset
$\mathcal R_0\subset \mathcal{RT}$ such that for every
$f\in\mathcal R_0$  holds that $\overline{\text{Per}_h(f)}=\Omega(f)=M$, where $\text{Per}_h(f)$
denotes the set of hyperbolic periodic points for $f$.
So,  we are left to prove the existence of dense pre-orbits for a generic
subset of robustly transitive local diffeomorphisms.
Using Corollary~\ref{cor:dense} it follows  that every $f\in\mathcal R_0$
satisfies property (3). This finishes the proof of the theorem.

\subsection{Proof of Proposition~\ref{prop2.1}}


Fix $f_0\in \mathcal{SRT}^*$. The first step is to recall that
$f_0$ is volume expanding, that is, $|\det
(Df_0)|\!>\!\sigma\!>\!1$. This follows from adapting the arguments
used by Bonatti, D\'iaz and Pujals ~\cite{BDP} in the invertible setting, as
we can see in the following theorem.

\begin{theorem}\cite[Theorem~4.3]{LP}\label{teo0}
Let $f$ be a $C^1$ local diffeomorphism and $U$ open set
in $M$ such that
$\Lambda_f(U)=\bigcap_{n\in\mathbb{Z}}f^n(\overline{U})$ is $C^1$
robustly transitive set and it has no splitting in a $C^1$ robust
way. Then $f$ is volume expanding.
\end{theorem}

Since there is no splitting in a $C^1$-robust way we can proceed
as in \cite[Lemma~6.1]{BDP} to prove that
there exists a $C^1$ local diffeomorphism $f\in \mathcal{SRT^*}$
arbitrarily close to $f_0$ and a periodic point $f^k(p)=p$  such that $Df^k(p)$ is an
homothety.
Moreover, since $f$ satisfies the hypothesis of the theorem above we deduce that $f$ is volume
expanding and, consequently, $p$ is periodic repelling and has a dense pre-orbit.
Since this is a robust property we deduce that there is an open
and dense subset $\mathcal {A}\subset \mathcal{SRT}^*$ such that
every $f\in \mathcal{A}$ has a repelling periodic point. In
particular, if $\mathcal{R}_0$ is given by Theorem~\ref{thm:preorbits}
(adapted $\mathcal{RT}^*$ with same proof) then every map in the  residual
subset $\mathcal{R}_0\cap \mathcal {A} \subset \mathcal{SRT}^*$ has a dense set of hyperbolic
periodic points and at least one periodic repelling point with dense pre-orbit.
This finishes the proof of the proposition.

\subsection{Proof of Theorem \ref{thm:measures}}

In this section we use the notion of zooming times 
to deduce the existence of interesting measures. More precisely, we prove the following:

\begin{theorem}\label{Theorem8769}
If a $C^1$ local diffeomorphism $f$ has a periodic
source with dense pre-orbit then there are uncountable many
invariant, ergodic and expanding measures with full support in $\Omega(f)$
and exponential decay of correlations.  In particular, if $f$
is transitive the measure support is total.
\end{theorem}



In order to prove the previous result we shall adapt some ideas from~\cite{Vilton}.
In fact, the conclusion above on the existence of many ergodic and expanding probability measures with
exponential decay of correlations has been established in  Proposition~9.3 and Theorem~5 of \cite{Vilton}
 for $C^{1+\alpha}-$maps admitting critical points. The $C^{1+\alpha}-$assumption is used there
 to obtain bounded distortion under the presence of critical points. Here we prove that this condition
 can be relaxed for local diffeomorphisms. Let us introduce the zooming times notion and a useful lemma before proving Theorem \ref{Theorem8769}.

\begin{defi}[Zooming contraction]\label{DefinitionZoomingContration} A sequence $\alpha=\{\alpha_n\}_{n\in\NN}$ of
functions $\alpha_n:[0,+\infty)\to[0,+\infty)$ is called a {\em
zooming contraction} if it satisfies: 
\begin{itemize}
\item $\alpha_n(r)<r,$ for every $r>0$ and $n\ge1$;
\item $\alpha_n(r)\le\alpha_n(\widetilde{r}),$  for every $0\le r\le \widetilde{r}$ and $n\ge1$;
\item $\alpha_n\circ\alpha_m(r)\le\alpha_{n+m}(r),$  for every $r>0$ and $n,m\ge1$;
\item $\displaystyle\sup_{0\le r\le1}\Big(\sum_{n=1}^{\infty}\alpha_n(r)\Big)<\infty$.
\end{itemize}
\end{defi}

Observe that an exponential backward contraction is an example of a zooming
contraction, $\alpha_n(r)=\lambda^n r$ with $0<\lambda<1$.
Let $\alpha=\{\alpha_n\}_{n}$ be a zooming contraction and
$\delta$ a positive constant.

\begin{figure}[h]
\includegraphics[scale=.25]{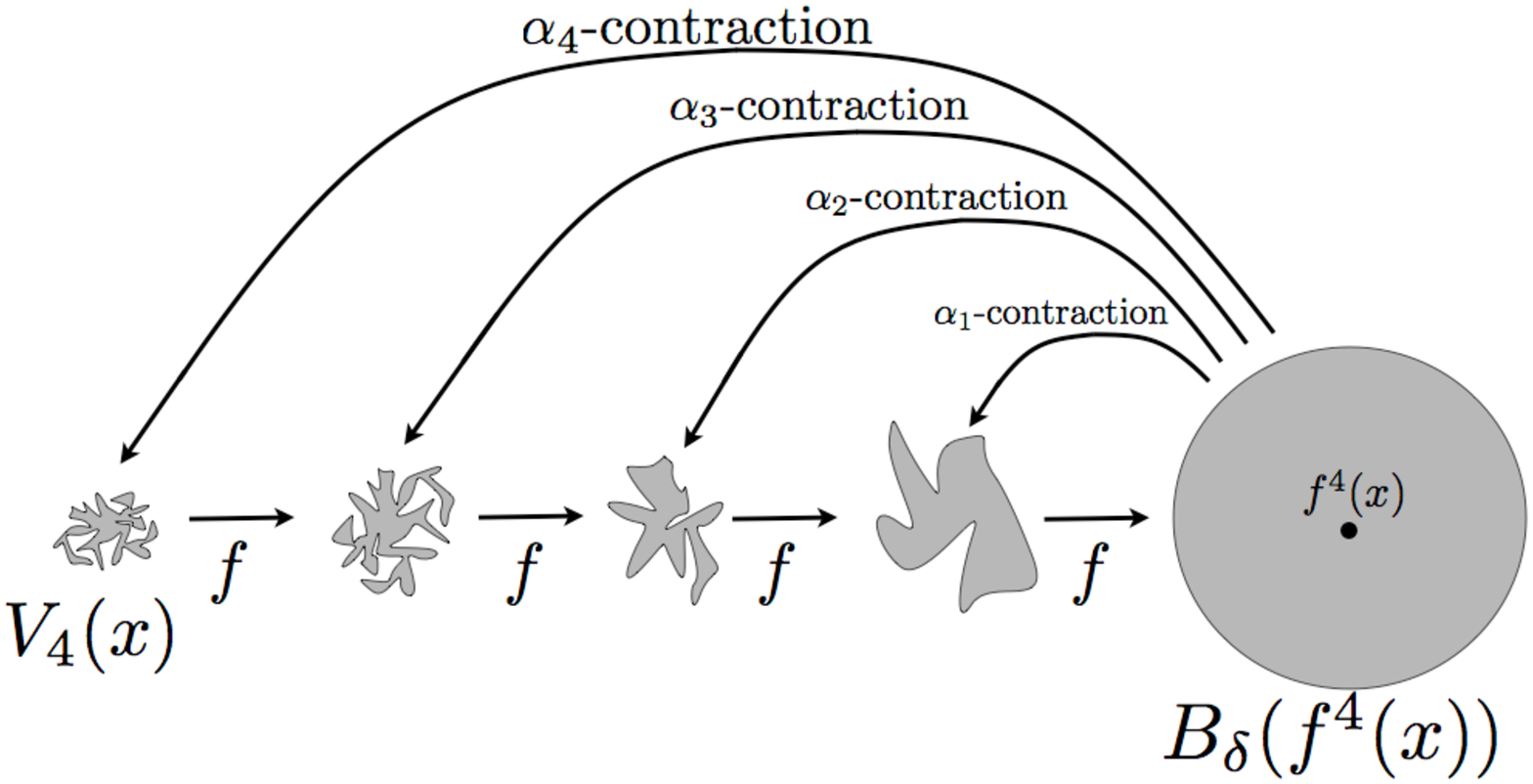}
\caption{A zooming time for
$x\in\tz_4(\alpha,\delta,f)$}\label{FigureZooming_time}
\end{figure}

\begin{defi}[Zooming times]
We say that $n\ge 1$ is a {\em $(\alpha,\delta)$-zooming time} for
$p\in M$, with respect to $f$, if there is a neighborhood $V_{n}(p)$ of $p$ such that
$f^{n}$ sends $\overline{V_n(p)}$  homeomorphically onto $\overline{B_{\delta}(f^{n}(p))}$ and
for all $x,y\in V_{n}(p)$ and $0\le j<n$
$$
\dist(f^j(x),f^j(y))\le\alpha_{n-j}\big(\dist(f^{n}(x),f^{n}(y))\big).
$$
\end{defi}

The ball $B_{\delta}(f^{n}(p))$ is called a {\em zooming ball} and
the set $V_{n}(p)$ is called a {\em zooming pre-ball}. Denote by
$\tz_n(\alpha,\delta,f)$ the set of points of $X$ for which $n$ is
an $(\alpha,\delta)$-zooming time.
A point $x\in M$ is called $(\alpha,\delta)$-zooming (with respect to $f$) if \begin{equation}\label{eqZZ}
\limsup_{n\to\infty}\frac{1}{n}\#\big\{1\le j\le n\,{;}\, x\in {\tz_j(\alpha,\delta,f)}\big\}>0.
\end{equation}
Moreover, a positively invariant set $\Lambda\subset X$ is called a {\em
$(\alpha,\delta)$-zooming set}, with respect to $f$, if  every $x\in\Lambda$ is  $(\alpha,\delta)$-zooming.

\begin{lemma}\label{LemmaExpandingPeriodicPreOrbit}
If $p$ is a periodic repeller with $\alpha_f(p)=M$ then there are $\ell\in\NN$, $\delta>0$ and $\lambda>1$ such that $\co_{f}^{-}(p)$ is a $(\alpha,\delta)$-zooming set with respect to $\widetilde{f}:=f^{\ell}$, where $\alpha=\{\alpha_n\}_n$ is given by
$\alpha_n(x)=(1/8)^n x$. Furthermore, there exists $p'\in\co_f^+(p)$ such that
$$
\bigg\{y\,\in\co_f^-(p')\colon \,\#\{j\in\NN\,;\,y\in Z_j(\alpha,\delta,\widetilde{f}\,)\text{ and }\widetilde{f}^j(y)=p'\}=\infty\bigg\}
$$
is dense in a neighborhood of $p'$.
\end{lemma}

\begin{proof}
Let $\gamma=\text{period(p)}$. Since $\co_f^+(p)$ is a finite set there exists $n_0\ge 1$ large so that $\log(\|(Df^{n\,\gamma}(q))^{-1}\|^{-1})>\log 32$
for all $n\ge n_0$ and every $q\in\co_f^+(p)$. Let $\delta>0$ be small enough such that,
for every $q\in\co_f^+(p)$ there is a neighborhood
$W(q)$ of the point $q$ satisfying $f^{n_0\gamma}(W(q))=B_{\delta}(q)$ and
$(f^{n_0\gamma}|_{W(q)})^{-1}$ is a $(e^{-\lambda_0})$-contraction, where $\lambda_0=\log 16$.

Given any $y\in\co_f^{-}(p)$ there exists a natural number $a$ such that $f^{ a}(y)= p$. Write $ a= (n_0\gamma)\,k_0+r_0$ with $0\le r_0 \le  (n_0\gamma)-1$ and $k_0\geq 0$. Let   $q_y=f^{ n_0\gamma-r_0}(p)\in\co^+_f(p)$.
Thus, $$f^{(k_0+1) n_0\gamma}(y)=f^{ n_0\gamma-r_0}(f^{k_0\, n_0\gamma+r_0}(y))=f^{ n_0\gamma-r_0}(p)=q_y.$$
We proceed to construct zooming neighborhoods of the point $y$.
Pick an open neighborhood  $W_0$ of $y$ so that $f^{(k_0+1) n_0\gamma}|_{W_{0}}$ is a diffeomorphism onto a
neighbourhood of $q_y$.
Note that since  $(f^{ n_0\gamma}|_{W(q_y)})^{-1}$ is a $(e^{-\lambda_0})$-contraction, there exists $j_0\in\NN$  such that for all $j\ge j_0$ it holds
$$
(f^{ n_0\gamma}|_{W(q_y)})^{-j}(B_{\delta}(q_y))\subset f^{(k_0+1) n_0\gamma}(W_0).
$$
For every $j\ge j_0$, let
$$U_{n_0\gamma(j+k_0+1)}(y):=(f^{(k_0+1){ n_0\gamma}}|_{W_0})^{-1}\circ (f^{ n_0\gamma}|_{W(q_y)})^{-j}(B_{\delta}(q_y)).$$
Therefore,
there is some $j_1\ge j_0$ such that $(f^{ n_0\gamma(j+k_0+1)}|_{U_{n_0\gamma(j+k_0+1)}(y)})^{-1}$ is a
$(e^{-\lambda_1})^{(j+k_0+1)}$-contraction for every $j\ge j_1$, where $\lambda_1=\log 8$.
By construction
\begin{equation}\label{eq9876533}
f^{ n_0\gamma(\,j+k_0+1)}(U_{n_0\gamma(j+k_0+1)}(y))=B_{\delta}(q_y)=B_{\delta}(f^{n_{0}\gamma (j+k_0+1)}(y)),
\end{equation}
and, consequently, $y\in Z_{j+k_0+1}(\alpha,\delta,f^{ n_0\gamma})$ for all $j\ge j_1$,
where the zooming sequence $\alpha=\{\alpha_n\}_n$ is given by
$\alpha_n(x)=e^{-\lambda_1\,n}x=(1/8)^n x$. This proves the first part of the lemma.

From now on set $\ell=n_0\gamma$. We obtained that $y\in Z_{j}(\alpha,\delta,f^{\ell})$ for all $j\ge j_1+k_0+1$
and, by \eqref{eq9876533}, it follows that
\begin{equation}\label{Eq867802h}
\#\{j\in\NN\,;\,y\in Z_j(\alpha,\delta,\widetilde{f}\,)\text{ and }\widetilde{f}^j(y)=q_y\}=\infty,
\end{equation}
where $\widetilde{f}=f^{\ell}$.

Consider the function $\varphi:\co_f^-(p)\to\co_f^+(p)$ defined as $\varphi(y)=q_y$, where $q_y\in\co_f^+(p)$
is chosen to satisfy \eqref{Eq867802h}. Since $\co_f^-(p)$ is dense in $M$ and one can write
$\co_f^-(p)=\varphi^{-1}(p)\uplus\cdots\uplus\varphi^{-1}(f^{\gamma-1}(p))$ as a disjoint union, there is $p'\in\co_f^+(p)$ such that $\varphi^{-1}(p')$ is dense in some open set $U\subset M$.
Let $x_0\in\varphi^{-1}(p')\cap U$ and $b\ge0$ such that $\widetilde{f}^{b}(x_0)=p'$. Since $f$ is a local homeomorphism, $\widetilde{f}^b(U)$ is a neighborhood of $p'$ and using that
$\varphi^{-1}(p')\subset \co_f^-(p)$ then it is also dense in $\widetilde{f}^b(U)$. This finishes the proof of the lemma.

\end{proof}

So in the remaining of this section we describe how to construct uncountable many ergodic and expanding
measures with full support and exponential decay of correlations.

\begin{proof}[Proof of Theorem~\ref{Theorem8769}]
Since the proof follows closely the one of \cite[Proposition~9.3~]{Vilton} we  give an outline of the
proof and focus on the main ingredients.
Assume that $f$ is a $C^1$ local diffeomorphism and $p$ is a periodic  source
with dense pre-orbit $\cO_f^-(p)$. 
Then, by the previous lemma
there exist $\ell\in\NN$ and
$\delta>0$, such that $\co_{f}^{-}(p)$ is a $(\alpha,\delta)$-zooming set
with respect to $\widetilde{f}:=f^{\ell}$ (in particular $\co_{\widetilde{f}}^-(p)$ is also a $(\alpha,\delta)$-zooming set for $\widetilde{f}$), where $\alpha=\{\alpha_n\}_n$ is the zooming sequence given by $\alpha_n(x)=(1/8)^n x$.
Moreover, changing $p$ for some $p'\in\co_f^+(p)$ if necessary, we have that $\co^Z(p):=\big\{y\in\co_{\widetilde{f}}^-(p)$ $;$
$\#\{j\in\NN\,;\,y\in Z_j(\alpha,\delta,\widetilde{f}\,)$ and $\widetilde{f}^j(y)=p\}=\infty$ $\big\}$ is dense in a neighborhood of $p$.
As $\sum_n\alpha_n(r)<r/4$, let $0<r<\delta/4$ be small such that $B_{r}(p)\subset\overline{\co^Z(p)}$.

So, the  $(\alpha,\delta)$-zooming nested ball with respect to $\widetilde{f}$, $\Delta=B^*_{r}(p)$,
is an open neighborhood of $p$ contained in $B_r(p)$ (see Definition~5.9 and also Lemma~5.12 in \cite{Vilton} for more details).
Furthermore, there is a dense set of points in $\Delta$ (the pre-orbit $\co^Z(p)\cap\Delta$) returning by $\widetilde{f}$ to $\Delta$ in a $(\alpha,\delta)$-zooming time.
In consequence, it follows from \cite[Corollary~6.6~]{Vilton} that there exists collection $\mathcal{P}$ of open connected subsets of $\Delta$ and an induced map $F:\Delta\to\Delta$ given by $F(x)=\widetilde{f}^{R(x)}(x)$ ($=f^{\ell\,R(x)}(x)$), with $\{R>0\}=\bigcup_{P\in\mathcal{P}}P$, such that $R$ is ``the first $(\alpha,\delta)$-zooming return time'' to $\Delta$ (see Definition~6.2~and~6.3 of \cite{Vilton}).

The function $R: \Delta \to\mathbb N$ is constant on elements of
$\mathcal P$ and $F$ satisfies the Markov property that $F(P)=\Delta$,
$F\mid_P$ is a $C^1$-diffeomorphism and $DF|_P>8$ for all $P\in\mathcal P$.

Now, for any sequence $a=(a_P)_{P\in\mathcal P}$ of real numbers satisfying
$0< a_P < 1$, $\sum_{P\in\cP} a_P=1$ and $\sum_{P\in\cP} a_P R(P)<\infty,$
let $\nu_a$ denote the Bernoulli measure which is defined on elements of the partition
$\cP^{(n)}=\bigvee_{j=0}^{n-1} F^{-j}\cP$ by
$$
\nu_a( P_0 \cap F^{-1} P_1 \cap \dots \cap F^{-(n-1)}P_{n-1} )
	= \prod_{j=0}^{n-1} a_{P_j}
$$
for all $n\ge 1$. It is not hard to check that $\nu_a$ is a $F$-invariant and ergodic probability measure and
 $\nu_a$ has constant Jacobian on cylinders (in fact $J_{\nu_a} F\mid_P =a_P$ for all $P\in\cP)$.
Now, using that $\int R \, d\nu_a=\sum_{P\in\cP} a_P R(P)<\infty$ then
$$
\mu_a=\frac{1}{\ell}\sum_{j=0}^{\ell-1}f_*^j\bigg(\frac{1}{\int R\, d\nu_a}
		\sum_{P\in\cP} \sum_{j=0}^{R(P)-1}\widetilde{f}^j_* (\nu_a\mid_P)\bigg)
$$
defines an $f$-invariant and ergodic probability measure. Moreover, $\mu_a\!\mid_{_\Delta} \ll \nu_a$ and
using that $\nu_a$ gives positive weight to open subsets of $\Delta$ then $P\subset \supp\mu_a$.
Since $f$ is transitive every positive invariant set with non empty interior is dense. Thus, $\mu_a$
has dense support and by compactness, the support is the whole manifold.

Furthermore, since each probability measure $\mu_a$ is ergodic and two ergodic measures either coincide or
are mutually singular, we deduce that there are uncountably many ergodic measures with full support, and those measures are expanding.
Indeed, $\nu_a$-almost every $x$ is $(\alpha,\delta)$-zooming, because $\nu_a\big(\bigcap_{j\ge0}F^{-j}(\{R>0\})\big)=1$ and we have
$$\limsup_{n}\frac{1}{n}\sum_{j=0}^{n-1}\log\big(\|(D\widetilde{f}(\widetilde{f}^j(x)))^{-1}\|^{-1}\big)>\log8>0$$
for every $x\in \bigcap_{j\ge0}F^{-j}(\{R>0\})$.
So, $\lim_{n}\frac{1}{n}\sum_{j=0}^{n-1}\log\big(\|(D\widetilde{f}(\widetilde{f}^j(x)))^{-1}\|^{-1}\big)>{\log8}>0$ for $\mu_a$-almost every $x$, since $\mu_a$ is $f$-invariant (and so, $\widetilde{f}$-invariant).
This implies that all Lyapunov exponents with respect to $\widetilde{f}$ (and also to $f$) are positive for $\mu_a$-almost every point.
 Finally, we notice that by \cite{Young}, $\mu_a$ has
exponential decay of correlations provided that
$$
\nu_a(R\ge n)
	=\sum_{k\ge n} \sum_{R(P)=k} a_P
$$
has exponential decay in $n$. Since the later property is satisfied for an uncountable many $(a_P)_{P\in\cP},$
this finishes the proof of the theorem.
\end{proof}


\section{Examples}\label{sec5}

\subsection{Existence of expanding measures with exponential decay}

We shall consider now an important class of robustly transitive
local diffeomorphisms introduced in \cite{LC,LP}.  Take $n\geq 2$
and $r\geq 1$. The following result holds:

\begin{theorem}\cite[Main Theorem]{LP}\label{teo4.1}
\emph{Let $f\in E^r(\mathbb{T}^n)$  be volume expanding map
 such that $\{w\in f^{-k}(x): k\in \mathbb{N}\}$ is dense for every $x\in \mathbb{T}^n$ and
satisfies the pro\-per\-ties:
\begin{enumerate}
\item[i)] There is an open set $U_0$ in $\mathbb{T}^n$ such that
      $f\!\!\mid_{U_0^c}$ is expanding and  $diam\!(U_0)\!\!<\!\!1$;
\item[ii)] There exists $0<\delta_0<diam_{int}(U_0^c)$ and there exists an open neighborhood $U_1$
        of $\overline{U}_0$ such that for every  arc $\gamma$ in $U_0^c$ with
      diameter larger than  $\delta_0,$ there is a point $y\in\gamma$ such that
      $f^k(y)\in U_1^c$ for any $k\geq 1$; and
\item[iii)] For every $z\in U_1^c,$ there exists $\bar{z}\in U_1^c$ such that $f(\bar{z})=z.$
\end{enumerate}
Then, for every $g$ $C^r-$close enough to $f$ all the points have dense pre-orbit,
that is,}  $\{w\in g^{-k}(x):
k\in \mathbb{N}\}$ is dense for every $x\in \mathbb{T}^n.$ In
particular, $f$ is $C^r$-robustly transitive.
\end{theorem}

The latter theorem essentially means that robust transitivity is obtained for local diffeomorphisms
whose pre-orbits are dense, if it is uniformly expanding in a definite region of the ambient space and
for every sufficiently large arc in this expanding region there exists a point whose forward orbit remains
in the expanding region.
%
%
In particular all periodic points have in fact dense pre-orbits.
In fact, it follows from \cite{LP} that hypotheses (2) and (3) above assure
the existence of a locally maximal expanding invariant
set $\Lambda_f$ which has a topological property of `separation'.
Roughly, every open set intersects $\Lambda_f$ after a finite number of iterates
which implies for future iterates the internal radius growth(IRG property):


\begin{lemma}\cite[Lemma 2.29]{LP} \label{lem4.1}
There exist $\mathcal{V}_2(f)$ and $R>0$ such that for every $g\in \mathcal{V}_2(f),$
if there is $x\in M$ such that
$g^n(x)\not\in U_0$ for every $n\geq 0,$ then there is $\varepsilon_0>0$
such that for every $0<\varepsilon<\varepsilon_0,$ there exists
$N=N(\varepsilon)\in\mathbb{N}$
such that $\mathbb{B}_R(g^N(x))\subset g^N(\mathbb{B}_{\varepsilon}(x)).$
\end{lemma}

Note that Lemma \ref{lem4.1} proves the robustness of IRG property, which is
fundamental to prove the density of the pre-orbit
of any point under the perturbed map.  For further details, see  \cite{LP}.
%
%
 After the  discussion above, we are now in condition to
present a large class of examples that illustrate our main
results. Let us consider $\mathcal F$ the class of $C^r$
endomorphisms $f$ in the $n$-dimensional torus $\mathbb{T}^n$
satisfying the following properties:
\begin{enumerate}
\item (volume expanding) There exists $\sigma\!\!>\!\!1$
    such that $|det(Df(x))|\!\!\ge\!\! \sigma$ for all $x\!\in \!\mathbb{T}^n$;
\item 
	Every point has dense pre-orbit;
\item There is an open set $U_0$ in $\mathbb{T}^n$ such that
      $f\!\!\mid_{U_0^c}$ is expanding and  $diam\!(U_0)\!\!<\!\!1$;
\item There exists a  locally maximal expanding invariant set $\Lambda_f\subset U_0^c$ with the topological
property of `separation' above. 
\end{enumerate}

Observe that every map satisfying the assumptions of Theorem \ref{teo4.1} belong to $\mathcal F$.
Indeed, property (4) can be shown to  be a consequence of the hypotheses (ii) and (iii) of Theorem~\ref{teo4.1} (
we refer the reader to \cite{LP} for details), while by property (2) all periodic points have dense pre-orbit.
Moreover, since the periodic points are dense, there are plenty of them in the expanding region. Hence, there is at least one periodic source in $\Lambda_f,$ because this set is invariant and expanding. Thus, every $f\in\mathcal F$ has at least one periodic source in $\Lambda_f$ with dense pre-orbit. 
Therefore, Theorem~\ref{Theorem8769} yields  \emph{for any $f\in
\mathcal F$ there are uncountable many ergodic, invariant and
expanding measures with full support and exponential decay of
correlations.}

\subsection{Example of robustly non-existence of splitting}

In this subsection we provide a large class of examples satisfying
our main results with robust non-existence of splitting from~\cite[Example 1]{LP}.
We provide here just the main ideas about the construction and show that this
class of maps satisfy the assumptions of our main results, and refer the reader to
\cite{LP} for details on the construction.


Let us consider a linear ex\-pan\-ding endomorphism
$\mathcal{E}:\mathbb{T}^n\rightarrow\mathbb{T}^n$ with $n\geq 2$ and
large topological degree $N$. There exists a Markov partition, 
denote the elements by $R_i$ with $1\leq i \leq N$.
For our purpose we can work with the initial partition, but if you wish to construct
a more general example doing a deformation of the Markov partition it is enough
to consider an isotopic map to the identity.

%

Pick $U_0$ an open set in $\mathbb{T}^n$
such that its convex hull $\widetilde{U}$ on the lift
 is contained in the
interior of $[0,1]^n$ (that implies $diam(U_0)<1$) and there exists at least one
$R_i$ contained in $U_0^c$.
Assume there exist $p\in U_0$ and
$q_i\in U_0^c$ expanding fixed points of  $\mathcal{E}$
with $1\leq i\leq n-1.$ These requirements are feasible since $\mathcal{E}$ has large degree.
%
Choose $\varepsilon>0$ small enough such that $\mathbb{B}_\varepsilon(q_i)\cap U_0=\emptyset$
and $\mathbb{B}_\varepsilon(q_i)\cap \mathbb{B}_\varepsilon(q_j)=\emptyset$ for all $i\neq j.$
Denote the  tangent space splitting as follows
$$T_x(\mathbb{T}^n)=\mathbb{E}_1^u(x)\prec \mathbb{E}_2^u(x) \prec \cdots\prec \mathbb{E}_{n-1}^u(x)\prec \mathbb{E}_n^u(x),$$
where $\prec$ denote that $\mathbb{E}_i^u(x)$ dominates the
expanding behavior of $\mathbb{E}_{i-1}^u(x).$
We proceed now to deform $\mathcal{E}$  by a smooth isotopy supported in $U_0\cup
(\bigcup\mathbb{B}_\varepsilon(q_i)).$ The perturbation $f$ is done in such a way that:
\begin{enumerate}
\item the continuation of $p$ goes through a pitchfork bifurcation,
       giving birth to two periodic points $r_1$ and $r_2$ in $U_0$ such that both are repeller,
       $p$ becomes a saddle point and $f$ still expands volume in $U_0$;
\item two expanding eigenvalues of $q_i$ associated to $\mathbb{E}_i^u(q_i)$ and $\mathbb{E}_{i+1}^u(q_i)$
      of $T_{q_i}(\mathbb{T}^n)$ become complex expanding eigenvalues for $f$.
        Thus, these two expanding subbundles are mixed obtaining
       $T^f_{q_i}(\mathbb{T}^n)=\mathbb{E}_1^u\prec \mathbb{E}_2^u \prec \cdots\prec
       \mathbb{F}_i^u\prec \cdots\prec\mathbb{E}_n^u,$ where $\mathbb{F}_i^u$ is two dimensional
       and correspond to the complex eigenvalues associated to $q_i$;
\item $f$ coincides with $\mathcal{E}$ in the complement of $U_0\!\cup\!(\bigcup\!\mathbb{B}_\varepsilon(q_i))$.
     Hence,  $f$ is expanding in $U_0^c.$
\end{enumerate}

\begin{figure}[h]
\begin{center}
\includegraphics[scale=0.5]{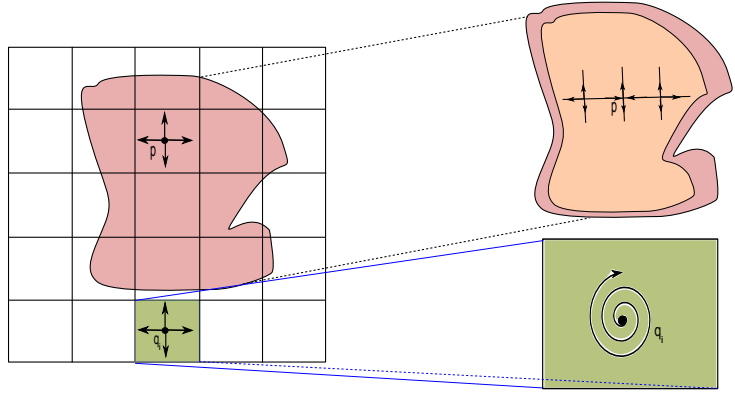}
\end{center}
\caption{\textit{$f$ isotopic to $\mathcal{E}$}}\label{graf3.2}
\end{figure}

Note that the e\-xis\-ten\-ce of these periodic points with complex eigenvalues prevent any non-trivial invariant
subbundle, and this construction is robust. Let us stress that the expanding region in these examples can be taken as small as desired. It can be shown that this class of examples satisfies the assumptions of Theorem~\ref{teo4.1} and,
therefore 
$f$ is a $C^1-$local diffeomorphism having a periodic source with dense pre-orbit. In particular, this class of examples admits uncountably many ergodic probability measures with exponential decay of correlations.

\section{Further comments}\label{sec6}

In this section, we address the problem of existence of relevant expanding measures for robustly transitive local diffeomorphisms
that admit some non-trivial invariant subbundle. First we introduce some notions.
Given a local diffeomorphism $f\in \text{Diff}_{\text{loc}}(M)$ and a compact forward invariant
set $\Lambda\subset M$  we say that $\Lambda$ admits a \emph{dominated splitting} if there exists a
continuous splitting $T_\Lambda M=E^1\oplus E^2$ and constants $C, a>0$ and $\la\in(0,1)$ such that
for all $x\in \Lambda$ and $n\in \mathbb N$:
\begin{itemize}
\item $Df(x) \, E^1_x = E^1_{f(x)}$ ($E^1$ is $Df$-invariant); and
\item the cone $\mathcal C^2_x=\{ u+v \in E^1_x\oplus E^2_x : \| u \| \le a \|v\|  \}$ satisfies
    the invariance condition $Df(x) (\mathcal C^2_x) \subset \mathcal C^2_{f(x)}$, and for all $v\in E^1_x \setminus\{0\}$
    and $w\in \mathcal C^2_x \setminus\{0\}$
    $$
    \frac{\|Df^n(x) v\|}{\|Df^n(x) w\|}  \leq C \lambda^n.
    $$
\end{itemize}
Since our previous results hold for maps whose tangent bundle does not admit  invariant subbundles
we now discuss the existence of expanding measures with full support in the presence of dominated splittings
that are robust by $C^1$-perturbations.

We say that a dominated splitting $T_\Lambda M=E^1\oplus E^2$ is  of \emph{expanding type}, namely if the
subbundle $E^1$ satisfies $\|(Df^n \mid_{E^1_x})^{-1}\| \leq C \lambda^n$ for all $x\in \Lambda$ and $n\ge 1$.
This implies that $f$ is uniformly expanding and so, by the theory developed by Sinai-Ruelle-Bowen, there are uncountable many $f$-invariant, ergodic and expanding measures with full support and exponential decay of correlations.

In a dual way,  we say that a dominated splitting $T_\Lambda M=E^1\oplus E^2$ is  of \emph{contracting type} if the
subbundle $E^1$ satisfies $\|Df^n \mid_{E^1_x}\| \leq C \lambda^n$ for all $x\in \Lambda$ and $n\ge 1$.
%
Note that if $T_\Lambda M=E^1\oplus E^2$ is a dominated splitting of contracting type then
expanding measures cannot exist due to the existence of invariant stable direction with uniform contraction
along the orbits. Hence, one could ask
wether there are uncountable ergodic and hyperbolic measures with total support and exponential decay
of correlations. The same strategy to prove Theorem~\ref{thm:measures} could answer the previous
question provided the existence of Markovian induced schemes for maps with a dense
non-uniformly hyperbolic set, which is an open question.

Finally, it remains to consider the case where $E^1$ is a center bundle  with non-uniform
expanding or contracting behavior and dominated by a a subbundle $E^2$ with uniform expansion.
Examples illustrating this situation and where there exists a periodic source with dense pre-orbit
can be found in \cite[subsection 5.3]{LP}. In particular such class of maps admit uncountable many invariant,
ergodic and expanding probability measures with full support and exponential decay of correlations.
We expect an analogous result as Theorem~\ref{thm:measures} to hold  for these type of maps.

%
%


\subsection*{Acknowledgments:}
The work was initiated after the Workshop on Dynamical Systems-Bahia 2011 at Universidade Federal da Bahia.
The authors are grateful to E. Pujals and L. D\'{i}az for
useful and encouraging conversations.
The first author is grateful to DMAT(PUC-Rio) and ICTP(math section) for the nice
environment provided during the preparation of this paper. The
first author was supported by CNPq and CDCHT-ULA project number AAA. The second and third
authors were also partially supported by CNPq and FAPESB. The second author was partially supported by the Balzan Research Project of J.Palis.


\end{document}